\documentclass[12pt, amslatex]{article}
\usepackage{amssymb, amsfonts, amsmath}
\usepackage{hyperref}
\usepackage{xcolor}
\hypersetup{colorlinks,	linkcolor={red!60!black},	citecolor={green!60!black}}
\newcommand{\qed}{\hskip 5mm \rule{2.5mm}{2.5mm}\vskip 10pt}
\newcommand{\R}{{\mathbb R}}

\newcommand{\N}{{\mathbb N}}
\newcommand{\Z}{{\mathbb Z}}

\newcommand*{\bigchi}{\mbox{\large$\chi$}}

\newcommand{\abs}[1]{ \left| #1 \right|}

\newcommand{\func}[2]{#1 \left( #2 \right)}

\newcommand{\parent}[1]{ \left( #1 \right)}
\newcommand{\brackets}[1]{ \left[ #1 \right]}

\newcommand{\setbuilder}[2]{ \left\{ #1 \mid #2 \right\}}
\newcommand{\2}[1]{\mathbb{#1}}

\newcommand{\floor}[1]{\left\lfloor #1 \right\rfloor}

\numberwithin{equation}{section}

\begin{document}
\newtheorem{theorem}{Theorem}[section]
\newtheorem{definition}[theorem]{Definition}
\newtheorem{lemma}[theorem]{Lemma}
\newtheorem{note}[theorem]{Note}
\newtheorem{corollary}[theorem]{Corollary}
\newtheorem{proposition}[theorem]{Proposition}
\renewcommand{\theequation}{\arabic{section}.\arabic{equation}}
\newcommand{\newsection}[1]{\setcounter{equation}{0} \section{#1}}
%%%%%%%%%%%% title %%%%%%%%%%%%%%%%%%%%%%%%%%%%%%%%
\title{On the Koopman-von Neumann convergence condition of Ces\`{a}ro means in Riesz spaces
\footnote{{\bf Keywords:} Ces\`{a}ro mean; Koopman-von Neumann; order convergence; Riesz spaces; conditional expectation operators; weak mixing.\
      {\em Mathematics subject classification (2010):} 46A40; 47A35; 37A25; 60F05.}
}
%%%%%%%%%%%%%%%%%%%%%%%%%%%%%%%%%%%%%%%%%%%%%%%%
\author{
 Jonathan Homann\\
 Wen-Chi Kuo\\
 Bruce A. Watson\footnote{Supported in part by the Centre for Applicable Analysis and
Number Theory and by National Research Foundation of South Africa grant IFR170214222646 with grant no. 109289.} \\ \\
% ${}^\sharp$ School of Mathematics\\
 School of Mathematics\\
 University of the Witwatersrand\\
 Private Bag 3, P O WITS 2050, South Africa }
\maketitle
\abstract{
\noindent
We extend the Koopman-von Neumann convergence condition on the Ces\`{a}ro mean to the context of a Dedekind complete Riesz space with weak order unit. As a consequence, a characterisation of conditional weak mixing is given in the Riesz space setting. The results are applied to convergence in $L^1$.
}

\parindent=0cm
\parskip=0.5cm
%%%%%%%%%%%%%%%%%%%%%%%%%%%%%%%%%%%%%%%%%%%%%%
\section{Introduction} \label{s: introduction}

The Koopman-von Neumann Lemma, as referred to by Petersen \cite[Section 2.6]{ergodic book} (see also Krengel \cite[Section 2.3]{krengel}, as well as Eisner, Farkas, Haase and Nagel \cite[Section 9.2]{EFHN}), characterises the convergence of the Ces\`{a}ro mean of a sequence of bounded, non-negative real numbers in terms of the existence of a convergent subsequence of the given sequence. Here, the subsequence is formed from the given sequence by the omission of a, so called, density zero set from ${\2N}_0$, the index set of the given sequence.

In this paper, we consider the order convergence of the Ces\`{a}ro mean of an order bounded, non-negative sequence in a Dedekind complete Riesz space with weak order unit. This requires a more sophisticated density zero concept. In particular, we introduce a density zero sequence of band projections which forms the foundation for the Koopman-von Neumann condition in Riesz spaces. When the Riesz space is the real numbers the characterisation presented here gives the classical Koopman-von Neumann convergence condition. 

As an application of the Koopman-von Neumann Lemma (Theorem \ref{koopman}), we give, in Section \ref{section:weak application}, a characterisation of conditional weak mixing in Riesz spaces. In Section \ref{section:application measurable}, as an example, we apply Theorem \ref{koopman} to characterise the order convergence of the Ces\`{a}ro mean of order bounded, non-negative sequences in $L^1$. 

This work supplements the development of stochastic processes in Riesz spaces of Grobler et al \cite{jensen paper}, Stoica \cite{stoica}, Azouzi et al \cite{azouzi}, Kuo, Labuschagne and Watson \cite{ergodic paper}, and mixing processes in Riesz spaces as considered in Kuo, Rogans and Watson \cite{KRW}, and Kuo, Vardy and Watson \cite{KVW}.

%%%%%%%%%%%%%%%%%%%%%%%%%%%%%%%%%%%%%%%
\section{Preliminaries}
 We refer the reader to Aliprantis and Border \cite{riesz book}, Fremlin \cite{fremlin}, Meyer-Nieberg \cite{MN-BL}, and Zaanen \cite{zaanen2} and \cite{zaanen}, for background in Riesz spaces and $f$-algebras.

We recall that, in a Riesz space, $E$, a sequence $\parent{f_n}_{{n \in {\2N}_0}}$ in $E$ converges to zero, in order, if and only if the sequence ${\parent{|f_n|}}_{n \in {\2N}_0}$ converges, in order, to zero in $E$. Further to this, in a Dedekind complete Riesz space, the  absolute order convergence of a sum implies the order convergence of the sum, see below.

\begin{lemma} \label{absolute-conditional convergence}
Let $\parent{f_n}_{n \in \2N_0}$ be a sequence in $E$, a Dedekind complete Riesz space, then order convergence of  
$\displaystyle{\sum_{k=0}^{\infty}|f_k|}$ implies the order convergence of 
$\displaystyle{\sum_{k=0}^{\infty}f_k}$.
\end{lemma}

\begin{proof}
Suppose that $\displaystyle{\sum_{k=0}^{n-1}|f_k|\to \ell}$, in order, as $n\to\infty$. Then $\ell$ is an upper bound for the increasing sequences
$\displaystyle{\left(\sum_{k=0}^{n-1}f_k^\pm\right)}$, which, from the Dedekind completeness of $E$, have order limits, say, $f^\pm$.  Thus,
$\displaystyle{\sum_{k=0}^{n-1}f_k=\sum_{k=0}^{n-1}f_k^+-\sum_{k=0}^{n-1}f_k^-}\to f^+-f^-$, in order, as $n\to\infty$.
\qed
\end{proof}

From the above lemma and \cite[Lemma 2.1]{KRodW2}, we have the following theorem.

\begin{theorem} \label{convergent sequence-absolute cesaro sum}
Let $E$ be a Dedekind complete Riesz space and $\parent{f_n}_{n \in \2N_0}$ a sequence in $E$ with $f_n \rightarrow 0$, in order, as $n \rightarrow \infty$, then
\(\displaystyle
\frac{1}{n} \sum_{k=0}^{n-1} \abs{f_k}
\rightarrow
0,
\)
in order, as $n \rightarrow \infty$.
\end{theorem}

\begin{corollary} \label{convergence implies cesaro convergence}
Let $E$ be a Dedekind complete Riesz space and $\parent{f_n}_{n \in \2N_0}$ a sequence in $E$ with order limit $f$, then,
as an order limit, we have
\(\displaystyle
\frac{1}{n} \sum_{k=0}^{n-1}f_k\to f
\)
in order, as $n\to\infty$.
\end{corollary}

\begin{proof}
Let $g_n:=f_n-f$ for each $n \in \2N_0$, then $\parent{|g_n|}_{n \in \2N_0}$ is order convergent to $0$, by assumption. Thus, by Theorem \ref{convergent sequence-absolute cesaro sum},
	\[
	0 \leq
	\abs{
	f- \frac{1}{n} \sum_{k=0}^{n-1} f_k
	}
	\leq
	\frac{1}{n} \sum_{k=0}^{n-1} \abs{g_k}
	\rightarrow
	0,
	\]
in order, as $n\to\infty$, and the result follows as $E$ is Archimedean.
\qed
\end{proof}

%%%%%%%%%%%%%%%%%%%%%%%%%%%%%%%%%%%%%

\section{Koopman-von Neumann Condition}
In \cite{ergodic book}, a subset $N$ of ${\2N}_0$ is said to be of density zero if
	\(\displaystyle
	\frac{1}{n}
	\sum_{k=0}^{n-1}
	\func{{\bigchi}_N}{k}
	\to
	0
	\)
as $n \to \infty$, where $\displaystyle \func{{\bigchi}_N}{k}=0$ if $k \in {\2N_0} \setminus N$ and $\displaystyle \func{{\bigchi}_N}{k}=1$ if $k \in N$. 
The Koopman-von Neumann Lemma \cite[Lemma 6.2]{ergodic book} asserts that if a sequence $(a_n)_{n \in \2N_0}$ of real numbers is non-negative and bounded, then $\displaystyle \frac{1}{n} \sum_{k=0}^{n-1} a_k \to 0$ as $n \to \infty$ if and only if there is $N$, a subset of $\2N_0$, of density zero, such that $a_n \to 0$ as $n \to \infty$, $n \in \2N_0 \setminus N$.

In order to extend the Koopman-von Neumann Lemma to sequences in a Riesz space, we define a density zero sequence of band projections as follows.

\begin{definition}[Density Zero Sequence of Band Projections] \label{subset of density zero definition}
A sequence $\displaystyle {\parent{P_n}}_{n \in {\mathbb{N}}_0}$ of band projections in a Riesz space $E$ is said to be of density zero if
	\(\displaystyle
	\frac{1}{n}
	\sum_{k=0}^{n-1}
	P_k
	\to 0
	\),
	in order as $n\to\infty$.
\end{definition}

With the above definition of density zero sequences of band projections, we can now give an 
analogue of the Koopman-von Neumann Lemma in Riesz spaces.

\begin{theorem}[Koopman-von Neumann] \label{koopman}
Let $E$ be a Dedekind complete Riesz space with weak order unit, say, $e$, and let
 $(f_n)_{n \in {\2N}_0}$ be an order bounded sequence in the positive cone of $E$, $E_+$, then
$\displaystyle{\frac{1}{n} \sum_{k=0}^{n-1} f_k\to 0},$ 
in order, as $n\to\infty$, if and only if there exists a density zero sequence of band projections ${\parent{P_n}}_{n \in \2N_0}$ such that 
$\parent{I-P_n} f_n \to 0$, in order, as $n \to \infty$.
\end{theorem}

\begin{proof}
Suppose that  $(f_n)_{n \in \2N_0} \subset E_+$ and there exists $g\in E_+$ with $f_n\le g$, for all $n\in \N_0$.

If $\parent{P_n}_{n\in\2N_0}$ is a density zero sequence of band projections with $(I-P_n)f_n \to0$, in order, as $n \to \infty$, then $P_n f_n \leq P_n g$, for all $n \in \2N_0$, and
	\begin{subequations}
	\begin{align}
	0 \leq \frac{1}{n} \sum_{k=0}^{n-1} f_k
	&=
	\frac{1}{n} \sum_{k=0}^{n-1} P_k f_k
	+
	\frac{1}{n} \sum_{k=0}^{n-1}
	\parent{I-P_k} f_k
	\label{that*}
	\\
	& \leq
	\left(\frac{1}{n} \sum_{k=0}^{n-1} P_k\right) g
	+
	\frac{1}{n} \sum_{k=0}^{n-1}
	\parent{I-P_k} f_k
	\label{this*}
	.
	\end{align}
	\end{subequations}
Since $(P_n)_{n\in\2N_0}$ is of density zero,
$\displaystyle
\frac{1}{n} \sum_{k=0}^{n-1}P_k
\to 0,
$
in order, as $n\to \infty$, giving
$\displaystyle
\parent{
\frac{1}{n} \sum_{k=0}^{n-1} P_k
}
g
\to 0,
$
in order, as $n \to \infty$. Furthermore, by Theorem \ref{convergent sequence-absolute cesaro sum}, as $(I-P_n)f_n \to 0$, in order, as $n \to \infty$ and $(I-P_n)f_n\geq0$, we have
$\displaystyle
\frac{1}{n} \sum_{k=0}^{n-1} \parent{I-P_k} f_k
\to
0,
$
in order, as $n \to \infty$. Thus, by \eqref{that*}-\eqref{this*}, as $E$ is Archimedean,
$\displaystyle
\frac{1}{n} \sum_{k=0}^{n-1} f_k
\to 0,
$
in order, as $n \to \infty$.

Conversely, suppose that $\displaystyle \frac{1}{n} \sum_{k=0}^{n-1} f_k \to 0$, in order, as $n \to \infty$. 
Let $P_{m,i}$ be the band projection onto the band generated by ${\parent{f_i-\frac{1}{m}e}}^+$, let $\displaystyle u_{m,j}:=\sup_{k \geq j} \frac{1}{k} \sum_{i=0}^{k-1} P_{m,i}e$ and let $R_{m,j}$ be the band projection onto the band generated by ${\parent{u_{m,j}-\frac{1}{m}e}}^+$.
As $0\le P_{m,i}e\le e$, we have that $0 \leq u_{m,j} \leq e$.
Further, since $P_{m,i}$ is increasing in $m$ for fixed $i$, it follows that $u_{m,j}$ is increasing in
$m$ for fixed $j$ and hence $u_{m,j}-\frac{1}{m}e$ is increasing in $m$ for fixed $j$, giving that 
$R_{m,j}$ is increasing in $m$ for fixed $j$.

Since $\{k\in\Z\,|\,k\ge j+1\}\subset \{k\in\Z\,|\,k\ge j\},$ it follows that 
$$u_{m,j+1}=\sup_{k \geq j+1} \frac{1}{k} \sum_{i=0}^{k-1} P_{m,i}e\le \sup_{k \geq j} \frac{1}{k} \sum_{i=0}^{k-1} P_{m,i}e=u_{m,j},$$
giving that $u_{m,j}$ is decreasing in $j$.
Hence, $R_{m,j}$ is decreasing in $j$, for fixed $m$.
We now show that, for fixed $m$,
$R_{m,j} \downarrow 0$, in order, as $j \to \infty$. Since $f_i\ge P_{m,i}f_i\ge\frac{1}{m}P_{m,i}e$ (as $P_{m,i}$ is the projection onto the band generated by ${\parent{f_i-\frac{1}{m}e}}^+$), we have
\begin{eqnarray}
\sup_{k\ge j} \frac{1}{k}\sum_{i=0}^{k-1}f_i
\ge\frac{1}{m}\sup_{k\ge j} \frac{1}{k}\sum_{i=0}^{k-1}P_{m,i}e=\frac{1}{m}u_{m,j}.\label{KvN-1}
\end{eqnarray} 
However, $u_{m,j}\ge R_{m,j}u_{m,j}\ge \frac{1}{m}R_{m,j}e$ (since $R_{m,j}$ is the band projection onto the band generated by ${\parent{u_{m,j}-\frac{1}{m}e}}^+$), so, by \eqref{KvN-1}, we have
\begin{eqnarray}
0\le\frac{1}{m^2}R_{m,j}e
\le
\frac{1}{m}u_{m,j}
\le
\sup_{k\ge j} \frac{1}{k}\sum_{i=0}^{k-1}f_i\to 0,\label{KvN-2}
\end{eqnarray} 
in order, as $j\to\infty$. Thus, for fixed $m$, $R_{m,j}\downarrow 0$, in order, as $j\to\infty$.

Observe that $R_{1,j}=0$, since $R_{1,j}$ is the band projection onto the band generated by ${\parent{u_{1,j}-e}}^+=0$.
Let $\displaystyle J_j:=\sup_{m\in\N} R_{m,j}$, for $j=0,1,\dots$, then $J_j\downarrow J$, say, in order, as $j\to \infty$ and
$R_{m+1,j} \parent{I-R_{m,j}}$, $m \in \N$, is a partition of $J_j$ for each $j=0,1,\dots$.
That is,
\begin{equation}
\parent{R_{m+1,j} \parent{I-R_{m,j}}} \wedge \parent{R_{n+1,j} \parent{I-R_{n,j}}}=0, \textrm{ for $m\ne n$},
\label{converse*1}
\end{equation}
and 
$\displaystyle \sum_{m=1}^{M-1} R_{m+1,j} \parent{I-R_{m,j}}=R_{M,j}\uparrow_M J_j$, in order, as $M\to \infty$, that is,
\begin{equation}
\sum_{m\in \2{N}}
R_{m+1,j} \parent{I-R_{m,j}}
=
J_j.
\label{dagger}
\end{equation}
Let 
\begin{equation}
Q_j:=\bigvee_{m \in \N} P_{m,j}R_{m+1,j} \parent{I-R_{m,j}}=\sum_{m \in \N} P_{m,j}R_{m+1,j} \parent{I-R_{m,j}},
\label{x}
\end{equation}
by \eqref{converse*1}.
Further, as $0 \leq P_{m,j} \leq I$,
\begin{equation}
Q_j=
\sum_{m\in \2{N}} P_{m,j}R_{m+1,j} \parent{I-R_{m,j}}
\leq
\sum_{m\in \2N} R_{m+1,j} \parent{I-R_{m,j}}
=
J_j.
\label{happy face}
\end{equation}
Hence, subtracting the left-hand side of \eqref{happy face} from the right-hand side of \eqref{happy face}, we obtain
\begin{equation}
\parent{I-Q_j}=(I-J_j)+(J_j-Q_j) =(I-J_j)+ \sum_{m \in \N} \parent{I-P_{m,j}}R_{m+1,j} \parent{I-R_{m,j}}.
\label{starhat}
\end{equation}
For $m>j+1$, $j \in \2N_0$, we have that $\frac{1}{j+1}>\frac{1}{m}$.
Now, from the definition of $u_{m,j}$, we have
\[
u_{m,j}
=
\sup_{k \geq j}
\frac{1}{k}
\sum_{i=0}^{k-1} P_{m,i}e
\geq
\frac{1}{j+1}
\sum_{i=0}^{j} P_{m,i}e
\geq
\frac{1}{j+1} P_{m,j}e
,
\]
thus $P_{m,j}u_{m,j}\ge \frac{1}{j+1}P_{m,j}e$.
Hence,
$P_{m,j}u_{m,j}\ge \parent{\frac{1}{j+1}-\frac{1}{m}}P_{m,j}e+\frac{1}{m}P_{m,j}e,$
giving
%$\displaystyle
\[
P_{m,j}\parent{u_{m,j}-\frac{1}{m}e}\ge \parent{\frac{1}{j+1}-\frac{1}{m}}P_{m,j}e,
\]
%$
so
%$\displaystyle
\[
P_{m,j} {\parent{u_{m,j}-\frac{1}{m}e}}^+
\geq
P_{m,j} \parent{u_{m,j}-\frac{1}{m}e}
\geq
\parent{\frac{1}{j+1}-\frac{1}{m}}P_{m,j}e
.
\]
%$
Hence,
${\parent{u_{m,j}-\frac{1}{m}e}}^+
\geq
\parent{\frac{1}{j+1}-\frac{1}{m}}P_{m,j}e.$
Further, $\frac{1}{j+1}-\frac{1}{m}>0$, so $P_{m,j}e$ is in the band generated by ${\parent{u_{m,j}-\frac{1}{m}e}}^+$, giving $R_{m,j}\geq P_{m,j}.$
Applying the above to $f_j$, we have 
\begin{equation} \label{new face}
R_{m,j}f_j \geq P_{m,j}f_j.
\end{equation}
Taking the supremum over $m\in\N$ in \eqref{new face} gives $J_jf_j\ge f_j$, since $\displaystyle \sup_{m\in \2N}R_{m,j}=J_j$ and $\displaystyle \sup_{m\in \2N} P_{m,j}f_j=f_j$.
Thus, $(I-J_j)f_j=0$, so, for $k \leq j$, applying \eqref{starhat} to $f_j$ gives
\begin{subequations}
\begin{align}
0
&
\leq
\parent{I-Q_j} f_j
\leq
\sum_{m \in \2N}
R_{m+1,j}
\parent{I-R_{m,j}}
\parent{I-P_{m,j}}
f_j
\label{3-10-a}
\\
&
\leq
\sum_{m \in \2N}
\frac{1}{m} R_{m+1,j}
\parent{I-R_{m,j}}e
,
\label{3-10-b}
\end{align}
\end{subequations}
where we have used that $\displaystyle \parent{I-P_{m,j}} f_j \leq \frac{1}{m} e$.
Now, from  \eqref{3-10-a}-\eqref{3-10-b}, for $k \leq j$,
\begin{subequations}
\begin{align}
0
&
\leq
\parent{I-Q_j} f_j
\leq
\sum_{m=1}^{k-1}
\frac{1}{m}
R_{m+1,j}
\parent{I-R_{m,j}}
e
+
\sum_{m=k}^{\infty}
\frac{1}{m}
R_{m+1,j}
\parent{I-R_{m,j}}
e
\\
&
\leq
\sum_{m=1}^{k-1}
R_{m+1,j}
\parent{I-R_{m,j}}
e
+
\frac{1}{k}
\parent{I-R_{k,j}}
e
\leq
R_{k,j}e+\frac{1}{k}e
.
\end{align}
\end{subequations}

So, taking the $\limsup$ as $j\to\infty$, for fixed $k \in \2N$, we obtain
	\begin{equation}
	0 \leq \limsup_{j \to \infty} (I-Q_j)f_j
	\leq
	\frac{1}{k}e + \limsup_{j\to\infty}R_{k,j}e=\frac{1}{k}e.
	\end{equation}
Thus, as $E$ is Archimedean, $\limsup_{j\to \infty} \parent{I-Q_j}f_j=0$ and
$\displaystyle
\parent{I-Q_j}f_j \to 0,
$
in order, as $j\to \infty$.

It now remains to show that
$\displaystyle
\frac{1}{n} \sum_{j=0}^{n-1}Q_je \to 0
,
$ in order, as $n\to\infty$. 
For $j<n$ and $m\ge M+1$, we have $R_{M+1,n}\le R_{m,j}$, giving $R_{M+1,n}(I-R_{m,j})=0$.
Thus, for $M,n\in\N$, we have
\begin{eqnarray*}
R_{M+1,n}(I-R_{M,n})\frac{1}{n}\sum_{j=0}^{n-1}Q_je
&=&\frac{1}{n}\sum_{j=0}^{n-1}R_{M+1,n}(I-R_{M,n})\bigvee_{m \in \N} P_{m,j}R_{m+1,j} \parent{I-R_{m,j}}e\\
&=&\frac{1}{n}\sum_{j=0}^{n-1} \bigvee_{m \le M} R_{M+1,n}(I-R_{M,n}) P_{m,j}R_{m+1,j} \parent{I-R_{m,j}}e\\
&\le&\frac{1}{n}\sum_{j=0}^{n-1}R_{M+1,n}(I-R_{M,n}) P_{M,j}e\\
&=&R_{M+1,n}(I-R_{M,n}) \frac{1}{n}\sum_{j=0}^{n-1}P_{M,j}e\\
&\le&R_{M+1,n}(I-R_{M,n}) u_{M,n}\\
&\le&\frac{1}{M}R_{M+1,n}(I-R_{M,n})e.
\end{eqnarray*}
Summing the above over $M\ge K$ gives
\begin{subequations}
\begin{align}
(J_n-R_{K,n})\frac{1}{n}\sum_{j=0}^{n-1}Q_je
&\leq
\sum_{M \geq K} \frac{1}{M} R_{M+1,n} \parent{I-R_{M,n}}e
\label{circled A1}
\\
&\le 
\frac{1}{K}(J_n-R_{K,n})e
\leq
\frac{1}{K} e.
\label{circled A2}
\end{align}
\end{subequations}
We recall, from \eqref{happy face}, that $Q_j\le J_j$, so if $j\ge \floor{\sqrt{n}}$ then $Q_j\le J_j\le J_{\floor{\sqrt{n}}}$ and $Q_jJ_{\floor{\sqrt{n}}}=Q_j$, that is,
$(I-J_{\floor{\sqrt{n}}})Q_j=0$.
Hence,
\begin{equation}
(I-J_{\floor{\sqrt{n}}})\frac{1}{n}\sum_{j=0}^{n-1}Q_je=\frac{1}{n}\sum_{j=0}^{\floor{\sqrt{n}}-1}Q_je\le \frac{1}{\sqrt{n}}e.
\label{circled B}
\end{equation}
Combining \eqref{circled A1}-\eqref{circled A2} and \eqref{circled B}, we have, for each $K\in\N$,
\begin{subequations}
\begin{align}
\frac{1}{n}\sum_{j=0}^{n-1}Q_je
&
=
\parent{
\parent{
I-J_{\floor{\sqrt{n}}}
}
+
\parent{
J_{\floor{\sqrt{n}}}
-J_n
}
+
\parent{
J_n-R_{K,n}
}
+
R_{K,n}
}
\frac{1}{n}
\sum_{j=0}^{n-1} Q_j e
\\
&
\leq
\frac{1}{\sqrt{n}}e+(J_{\floor{\sqrt{n}}}-J_n)e+\frac{1}{K}e+R_{K,n}e
.
\label{KvN-bound}
\end{align}
\end{subequations}
Here, $J_n$ and $J_{\floor{\sqrt{n}}}$ both converge, in order, to $J$, so $(J_{\floor{\sqrt{n}}}-J_n)e$ converges, in order, to $0$ as $n\to\infty$, as does $R_{K,n}e$.
Thus, taking the limit supremum as $n\to\infty$ in (\ref{KvN-bound}) gives 
\begin{equation}
0\le \limsup_{n\to\infty}\frac{1}{n}\sum_{j=0}^{n-1}Q_je\le \frac{1}{K}e
\end{equation}
for each $K\in\N$. Hence,
$$\lim_{n \to \infty} \frac{1}{n}\sum_{j=0}^{n-1}Q_je =0,$$
in order, as $E$ is Archimedean.
\qed
\end{proof}

{\bf Note:} If $(f_n)_{n \in \2N_0}$ in Theorem \ref{koopman} is not assumed to be order bounded, but
\(\displaystyle
\frac{1}{n} \sum_{k=0}^{n-1} f_k\to 0,
\)
in order, as $n\to\infty$, then it still follows that there is a density zero sequence of band projections, $\parent{P_n}_{n \in \2N_0}$, such that
$\displaystyle
(I-P_n)f_n
\to 0,
$
in order, as $n \to \infty$, but one cannot conclude boundedness of $(P_n f_n)_{n \in \2N_0}$.
Further, the converse need not hold. In particular, if $(f_n)_{n \in \2N_0}$ is not bounded and there is a density zero sequence of band projections $\parent{P_n}_{n \in \2N_0}$ such that
$\displaystyle
(I-P_n)f_n
\to
0,
$
in order, as $n \to \infty$, then one cannot conclude that 
$\displaystyle
\frac{1}{n} \sum_{k=0}^{n-1} f_k$ is order convergent to $0$ as $n\to\infty$.

For example, working in the classical case of  $E=\R$, we have $g^p_j=\left\{\begin{array}{ll} n,&j=\floor{n^p}, n\in\N,\\ 0,&\mbox{otherwise}\end{array}\right.$ is unbounded and
$N=\{\floor{n^p}\,|\,n\in\N\}$ is a set of density zero for all $p>1$. Here $g^p_j=0\to 0$ for $j\in\N\backslash N$ as $j\to\infty$ but
$\displaystyle \frac{1}{n} \sum_{k=0}^{n-1} g^p_k$ is convergent to $0$ for $p>2$, convergent to a non-zero value for $p=2$ and divergent to $\infty$ for $1<p<2$.

%%%%%%%%%%%%%%%%%%%%%%%%%%%%%%%%%%%%
\section{Application to Weak Mixing} \label{section:weak application}

Let $\parent{\Omega,\3A,\mu}$ be a probability space, that is, $\Omega$ is a set, $\3A$ is a $\sigma$-algebra of subsets of $\Omega$ and $\mu$ is a measure on $\3A$ with $\func{\mu}{\Omega}=1$. The mapping $\tau : \Omega \to \Omega$ is called a measure preserving transformation if $\func{\mu}{{\tau}^{-1} A}=\func{\mu}{A}$, for each $A \in \3A$, in which case $\parent{\Omega,\3A,\mu,\tau}$ is called a measure preserving system. Further details may be found in \cite{EFHN, krengel, ergodic book}.

The measure preserving system $\parent{\Omega,\3A,\mu,\tau}$ is said to be weakly mixing if
\begin{equation}
\frac{1}{n} \sum_{k=0}^{n-1}
\abs{
\func{\mu}{\func{{\tau}^{-1}}{A} \cap B}
-
\func{\mu}{A} \func{\mu}{B}
}
\to 0,
\end{equation}
in order, as $n\to\infty$, for each $A,B \in \3A$.
To give a Riesz space analogue of a measure preserving system and weak mixing, we begin by defining a conditional expectation operator on a Riesz space. For more details, see \cite{expectation paper}.

\begin{definition}
Let $E$ be a Riesz space with weak order unit. A positive order continuous projection $T \colon E \rightarrow E$, with range, $\func{\3R}{T}$, a Dedekind complete Riesz subspace of $E$, is called a conditional expectation operator if $Te$ is a weak order unit of $E$ for each weak order unit $e$ of $E$.
\end{definition}

The Riesz space analogue of a measure preserving system is introduced in the following definition.

\begin{definition} \label{system definition}
Let $E$ be a Dedekind complete Riesz space with weak order unit, say, $e$, and let $T$ be a conditional expectation operator on $E$ with $Te=e$. If $S$ is an order continuous Riesz homomorphism on $E$ with $Se=e$ and $TSPe=TPe$ for each band projection $P$ on $E$, then $\parent{E,T,S,e}$ is called a conditional expectation preserving system.
\end{definition}

By Freudenthal's Spectral Theorem (\cite[Theorem 33.2]{zaanen}), the condition $TSPe=TPe$ for each band projection $P$ on $E$ in the above definition is equivalent to $TSf=Tf$ for all $f \in E$.

For $f \in E$, various convergence results for 
$\displaystyle
{\parent{
\frac{1}{n} \sum_{k=0}^{n-1} S^k f
}}_{n \in {\2N}_0}
$
were considered in \cite[Theorems 3.7 and 3.9]{ergodic paper}, most notably,  generalisations of Birkhoff's ergodic theorems to Riesz spaces.

We recall that 
$E_e= \setbuilder{f \in E}{\abs{f} \leq ke \textrm{ for some $k \in {\2R}_+$}},$ 
 the subspace of $E$ of $e$ bounded elements of $E$, 
is an $f$-algebra, see \cite{azouzi, venter, zaanen2}. Further, if $T$ is a conditional expectation operator on $E$ with $Te=e$, then $T$ is also a conditional expectation operator on $E_e$, since, if $f \in E_e$, then $\abs{f} \leq ke$, giving $\abs{Tf} \leq T \abs{f} \leq Tke=ke$.
The $f$-algebra structure on $E_e$ gives $Pe \cdot Qe = PQe$ for all band projections $P$ and $Q$ on $E$ (here, $\cdot$ represents the $f$-algebra multiplication on $E_e$). The linear extension of this multiplication and use of order limits extends this multiplication to $E_e$.

We are now in a position to define conditional weak mixing on a Riesz space with a conditional expectation operator and weak order unit.

\begin{definition}[Weak Mixing] \label{weak mixing definition}
The conditional expectation preserving system $\parent{E,T,S,e}$ is said to be weakly mixing if, for all band projections $P$ and $Q$ on $E$,
\begin{equation} \label{weak mixing equation}
\frac{1}{n} \sum_{k=0}^{n-1} \abs{\func{T}{\parent{S^kP}Qe}-TPe \cdot TQe}\to 0,
\end{equation}
in order, as $n\to\infty$.
\end{definition}

We note that for $E=\func{L^1}{\Omega,\3A,\mu}$, the band projections on $E$ are of the form $P_Af= {\bigchi}_Af$, for $f \in E$ and $A \in \3A$. Definition \ref{weak mixing definition} now gives a conditional weak mixing condition on $E$, conditioned by $T= {\2E} \brackets{\cdot,\3B}$, for $\3B \subseteq \3A$ a sub-$\sigma$-algebra. If $\3B = \setbuilder{\Omega \setminus C}{C \in \3A, \func{\mu}{C}=0} \cup \setbuilder{C \in \3A}{\func{\mu}{C}=0}$, then conditional weak mixing coincides with the weak mixing on $\parent{\Omega, \3A,\mu}$.

In a Riesz space, $E$, for $u \in E_+$,  $p \in E_+$ is called a component of $u$ if $u \wedge \parent{e-p}=0$, see \cite[Chapter 32, pg 213]{zaanen}. Furthermore, if $E$ has the principal projection property, then $p$ is a component of $u$ if and only if $p=P_fu$, for some principal projection $P_f$, \cite[Theorem 32.7]{zaanen}. If $E$ has a weak order unit, say, $e$, then any $s \in E$ for which there exist pairwise disjoint components $p_1,...,p_n$ of $u$ and real numbers ${\alpha}_1,...,{\alpha}_n$ such that $s= \sum_{k=1}^n {\alpha}_k p_k$ is called an $e$-step function (it is permitted for one or several of the real numbers of components of $e$ to be zero). Notice that if $E$ has the principal projection property, then there exist principal band projections $P_1,...,P_n$ such that $s=\sum_{k=0}^n {\alpha}_k P_ke$, where the band projections are pairwise disjoint. Consequently, if $E$ is a Dedekind complete Riesz space with weak order unit, say, $e$, then $B_e=E \supseteq E_e$ and any $e$-step function $s \in E$ can be represented by $s=\sum_{k=1}^n {\alpha}_k P_k e$, as discussed, where $B_e$ is the principal band generated by $e$.

We recall, \cite[Chapter 10, pg 49]{zaanen}, that if $E$ is a Riesz space and ${\parent{x_n}}_{n \in \2N} \subset E$ converges to $x \in E$, we say that $x_n$ converges to $x$ $u$-uniformly if for given $0<u \in E$ and each $0< \epsilon \in \2R$, there is some $N_{\epsilon} \in \2N$ such that $\abs{x_n-x}< \epsilon u$ whenever $n \geq N_{\epsilon}$.

\begin{theorem} \label{weak mixing big theorem}
Given the conditional expectation preserving system $\parent{E,T,S,e}$, then the following statements are equivalent.
\begin{enumerate}
\item
\label{weak mixing big theorem 1}
$\parent{E,T,S,e}$ is weakly mixing.
\item For  all $f,g \in E_e$, as $n\to \infty$, we have, in order, that
\label{weak mixing big theorem 2}
\[
\frac{1}{n} \sum_{k=0}^{n-1} \abs{\func{T}{\parent{S^kf} \cdot g}-Tf \cdot Tg} \rightarrow 0.
\]
\item
\label{weak mixing big theorem 3}
For each pair of band projections $P$ and $Q$ on $E$, there is a sequence of density zero band projections, $\displaystyle {\parent{R_n}}_{n \in {\N}_0},$ in $E$ such that
\[
\parent{I-R_n}
\abs{\func{T}{\parent{S^nP}Qe}-\func{T}{Pe} \cdot \func{T}{Qe}}
\to 0
,
\]
in order, as $n \rightarrow \infty$.
\end{enumerate}
\end{theorem}

\begin{proof}
{\bf ({\ref{weak mixing big theorem 1}})$\Rightarrow$({\ref{weak mixing big theorem 2}}):} 
Suppose that $\parent{E,T,S,e}$ is weakly mixing, let $\alpha, \beta \in \2R$ and let $P$ and $Q$ be band projections on $E$, then
\begin{subequations}
\begin{align*}
&
\frac{1}{n} \sum_{k=0}^{n-1}
\abs{
\func{T}{\parent{
\func{S^k}{\alpha P e}
} \cdot
\parent{\beta Q e}
}
-
\func{T}{
\alpha P e
}
\cdot
\func{T}{
\beta Q e
}
}
\\
& =
\abs{\alpha \beta}
\frac{1}{n}
\sum_{k=0}^{n-1}
\abs{
\func{T}{\parent{
\func{S^k}{P}}Qe
}
-
\func{T}{Pe} \cdot \func{T}{Qe}
}
\to 0,
\end{align*}
\end{subequations}
in order, as $n \to \infty$, by \eqref{weak mixing equation}.

Let $s,t \in E$ be $e$-step functions with $\displaystyle s= \sum_{i=1}^m {\alpha}_i P_i e$ and $\displaystyle t=\sum_{j=1}^r {\beta}_j Q_j e$, where $P_i$ and $Q_j$ are band projections on $E$ and $\alpha_i$ and $\beta_j$ are real numbers, $i=1,\dots,m$, $j=1,\dots,r$, then
\begin{subequations}
\begin{align}
&\frac{1}{n}
\sum_{k=0}^{n-1}
\abs{
\func{T}{\parent{
\func{S^k}{s}} \cdot t
}
-
\func{T}{s} \cdot \func{T}{t}
}
\\
&
\leq
\sum_{i=1}^m \sum_{j=1}^r \abs{{\alpha}_i {\beta}_j}
\frac{1}{n}
\sum_{k=0}^{n-1}
\abs{
\func{T}{\parent{\func{S^k}{P_i}}Q_je}
-
\func{T}{P_ie} \cdot \func{T}{Q_je}
}
\to 0,
\label{use this}
\end{align}
\end{subequations}
in order, as $n \to \infty$.

By Freudenthal's Spectral Theorem, \cite[Theorem 33.2]{zaanen}, $f,g \in E_e$ can be expressed as the $e$-uniform order limits of sequences, say ${\parent{s_i}}_{i \in \2N},{\parent{t_j}}_{j \in \2N}$, of $e$-step functions in $E_e$. 
There is now $K>0$ so that $\abs{s_i},\abs{t_j}, \abs{f},\abs{g} \leq Ke$, for each $i,j \in \2N$. This implies $T\abs{s_i}, T{\abs{t_j}}, T{\abs{f}}, T{\abs{g}} \leq Ke$, for each $i,j \in \2N$.
For each $\epsilon>0$ there is $N_\epsilon\in\2N$ so that $\abs{s_i-f} \leq \epsilon e$ and $\abs{t_j-g} \leq \epsilon e$ for each $i,j \geq N_{\epsilon} \in \2N$, then $T{\abs{s_i-f}} \leq \epsilon e$ and $T{\abs{t_j-g}} \leq \epsilon e$ for each $i,j \geq N_{\epsilon}$. Let $b_k:=\abs{\func{T}{{\func{S^k}{f}} \cdot g}-T{f} \cdot T{g}}$ and $b_{k,i,j}:=\abs{\func{T}{{\func{S^k}{s_i}} \cdot t_j}-T{s_i} \cdot T{t_j}}$, then
\begin{subequations}
\begin{align*}
\abs{b_{k,i,j}-b_k}
& = \left|\abs{\func{T}{{\func{S^k}{s_i}} \cdot t_j}-Ts_i \cdot Tt_j}-  \abs{\func{T}{{\func{S^k}{f}} \cdot g}-Tf \cdot Tg}\right|\\
& \le \left|{\func{T}{{\func{S^k}{s_i}} \cdot t_j}-Ts_i \cdot Tt_j}-  {\func{T}{{\func{S^k}{f}} \cdot g}+Tf \cdot Tg}\right|\\
&
\leq\abs{\func{T}{{\func{S^k}{f}}\cdot g
-{\func{S^k}{s_i}}\cdot t_j}}+\abs{Tf \cdot Tg -Ts_i \cdot Tt_j}\\
&
\leq
T{\abs{\left(\func{S^k}{f-s_i}\right)\cdot g}}
+T{\abs{\func{S^k}{s_i}\cdot (g-t_j)}}
%\\   &\quad 
+T{\abs{f-s_i}} \cdot T{\abs{g}}
+
T{\abs{s_i}} \cdot T{\abs{g-t_j}}
\\
&
\leq
Ke \cdot T{\abs{\func{S^k}{f-s_i}}}
+
Ke \cdot T{\abs{g-t_j}}
+
T{\abs{f-s_i}} \cdot Ke + Ke \cdot T{\abs{g-t_j}}
\\
&
\leq
4K \epsilon e.
\end{align*}
\end{subequations}
Hence, $b_k \leq b_{k,i,j}+4K \epsilon e$, so
\[
\frac{1}{n} \sum_{k=0}^{n-1} b_k
\leq
\frac{1}{n} \sum_{k=0}^{n-1} b_{k,i,j} + 4K \epsilon e,
\]
for all $n,i,j \in \2N$ with $i,j \geq N_{\epsilon}$. In particular,
\[
0 \leq \frac{1}{n} \sum_{k=0}^{n-1} b_k
\leq \frac{1}{n} \sum_{k=0}^{n-1} b_{k,N_{\epsilon},N_{\epsilon}}+ 4K \epsilon e
.
\]
By \eqref{use this}, $\displaystyle \frac{1}{n} \sum_{k=0}^{n-1} b_{k,N_{\epsilon},N_{\epsilon}} \to 0$, in order, as $n \to \infty$, so
\[
0
\leq
\limsup_{n \to \infty} \frac{1}{n} \sum_{k=0}^{n-1} b_k
\leq
4K \epsilon e,
\]
for all $\epsilon > 0$, implying that $\displaystyle \frac{1}{n} \sum_{k=0}^{n-1} b_k \to 0$, in order.

{\bf ({\ref{weak mixing big theorem 2}})
$\Rightarrow$
({\ref{weak mixing big theorem 1}}):} 
Choosing $f=Pe$ and $g=Qe$,  the result follows directly.

{\bf ({\ref{weak mixing big theorem 1}})$\Leftrightarrow$({\ref{weak mixing big theorem 3}}): }
Choosing $f_n=\abs{\func{T}{\parent{S^nP}Qe}-TPe \cdot TQe}$,  the result follows from Theorem~\ref{koopman}.
\qed
\end{proof}
%%%%%%%%%%%%%%%%%%%%%%%%%%%%%%%%%%%%
\section{Application to Measurable Processes}
\label{section:application measurable}

If we consider the Riesz space $E$  of equivalence classes of almost everywhere identical functions in $L^1(\Omega,{\mathcal{A}},\mu)$, where $\mu$ is a finite measure (the case of $\mu$ $\sigma$-finite is an easy extension of this case), then $E$ is a Dedekind complete Riesz spaces under a.e. pointwise ordering and $\bf{1}$, the equivalence class of the constant function with value $1$, is a weak order for $E$. Here the band projections, $P$, on $E$ are multiplication by the characteristic function of measurable sets, i.e., $P$ is of the form $Pf=\bigchi_A f, f\in E,$ for $A\in{\mathcal{A}}$.
If $(f_n)_{n \in {\2N}_0}$ is a non-negative order bounded sequence in $E$, i.e., $f_n\ge 0$ a.e. and there exists $g\in  L^1(\Omega,{\mathcal{A}},\mu)$ so that $f_n\le g$ a.e. for all $n$, then, by Theorem \ref{koopman}, 
\begin{equation}\label{cesaro-f}
 \frac{1}{n}\sum_{j=0}^{n-1}f_j\to 0, \quad\mbox{in order, as}\quad n\to\infty,
\end{equation}
if and only if there is a density zero sequence of band projections $(P_n)_{n \in {\2N}_0}$ such that $(I-P_n)f_n\to 0$, in order, as $n\to\infty$, i.e., there is a sequence of measurable sets $(A_n)_{n \in {\2N}_0}$ with 
\begin{equation}\label{density-0}
 \frac{1}{n}\sum_{j=0}^{n-1}\bigchi_{A_j}\to 0\quad\mbox{in order as}\quad n\to\infty,
\end{equation}
with $\bigchi_{\Omega\backslash A_n}f_n\to 0$ in order as $n\to\infty$.

Here, order convergence is the same as a.e. pointwise convergence, further, as $\displaystyle{\frac{1}{n}\sum_{j=0}^{n-1}f_j\le g}$,  $\displaystyle{\frac{1}{n}\sum_{j=0}^{n-1}\bigchi_{A_j}\le \bf{1}}$ and $\bigchi_{\Omega\backslash A_n}f_n\le g$, Lebesgue's Dominated Convergence Theorem is applicable. Thus we have the following.

\begin{corollary}
If $(f_n)_{n \in {\2N}_0}$ is a non-negative sequence in $L^1(\Omega,{\mathcal{A}},\mu)$, where $\mu$ is a finite measure, and there exists  $g\in L^1(\Omega,{\mathcal{A}},\mu)$ so that $f_n\le g$ a.e. for all $n \in {\2N}_0$, then 
$\displaystyle{\frac{1}{n}\sum_{j=0}^{n-1}f_j\to 0}$ as $n\to\infty$,
if and only if there is a sequence of measurable sets $(A_n)_{n \in {\2N}_0}$ with
$\displaystyle{
 \frac{1}{n}\sum_{j=0}^{n-1}\bigchi_{A_j}\to 0}$
and $\bigchi_{\Omega\backslash A_n}f_n\to 0$ as $n\to\infty$. Here the limits can be taken as either a.e. pointwise or in norm.
\end{corollary}

Proceeding as in \cite[Section 5]{KRW}, the weak mixing of Section \ref{section:weak application} can be carried over to give conditional weak mixing in measure spaces and a characterization thereof.
%%%%%%%%

\end{document}